\documentclass[12pt]{amsart}
\usepackage[english]{babel}
\usepackage{amsfonts}
\usepackage{amsthm}
\usepackage{amsbsy}
\usepackage{amssymb}
\usepackage{graphicx}
\usepackage{eso-pic}
\usepackage{tikz}
\usepackage{xfrac}
\usepackage{float}
\usepackage{faktor}
\usepackage[bookmarks=true,colorlinks=true,urlcolor=black,linkcolor=black,citecolor=black,pagebackref=false]{hyperref}
\usepackage{resizegather}
\usepackage{bm}
\usepackage{bookmark}
\usepackage{amsmath}
\usepackage{subfigure}
\usepackage{hyperref}
\usepackage{scalerel,stackengine}
\usepackage{epigraph}
\usetikzlibrary{matrix,arrows}
\usepackage[letterpaper,margin=1.0in,top=1.25in,bottom=1.25in]{geometry}
\usepackage{microtype}
\usepackage{todonotes}
\usepackage[justification=centering]{caption}
\usepackage{tikz-cd}
\makeatletter
\newcommand\footnoteref[1]{\protected@xdef\@thefnmark{\ref{#1}}\@footnotemark}
\makeatother

\newtheorem{lemma}{Lemma}[section]
\newtheorem{thm}[lemma]{Theorem}
\newtheorem{prop}[lemma]{Proposition}
\newtheorem{cor}[lemma]{Corollary}
\newtheorem*{cor*}{Corollary}

\theoremstyle{definition}
\newtheorem{defn}[lemma]{Definition}

\newtheorem{example}[lemma]{Example}

\theoremstyle{remark}
\newtheorem{rem}[lemma]{Remark}

\stackMath
\newcommand\reallywidehat[1]{%
\savestack{\tmpbox}{\stretchto{%
  \scaleto{%
    \scalerel*[\widthof{\ensuremath{#1}}]{\kern-.6pt\bigwedge\kern-.6pt}%
    {\rule[-\textheight/2]{1ex}{\textheight}}
  }{\textheight}%
}{0.5ex}}%
\stackon[1pt]{#1}{\tmpbox}%
}
\author{Leonardo Ferrari}
\address{\parbox{\linewidth}{Institut de Math\'ematiques, Universit\'e de Neuch\^atel, 2000 Neuch\^atel, Suisse / Switzerland}}
\email[]{leonardocpferrari at gmail dot com}

\title[Coloured Symmetries of $QHS^3$ that Cover Right-Angled Polytopes]{Coloured Symmetries of Rational Homology $3$-Spheres that Cover Right-Angled Polytopes}

\begin{document}

\begin{abstract}
In the present paper we study hyperbolic manifolds that are rational homology $3$--spheres obtained by colouring of right--angled polytopes. We study the existence (or absence) of different kinds of symmetries of rational homology spheres that preserve the tessellation of the manifold into polytopes. We also describe how to create colourings with given symmetries.
\end{abstract}


\maketitle

\section{Introduction}

Colourings of simple $3$--dimensional hyperbolic polytopes with elements of $\mathbb{Z}^3_2$ were first used by Vesnin \cite{Ves87} in order to produce closed hyperbolic $3$--manifolds from the so--called L\"obell polyhedra. 

Later on, Davis and Januszkiewicz \cite{DJ91} used a similar idea in order to produce complex $n$--manifolds with torus actions. The fixed point set of the conjugation involution of any such manifold produces a real $n$--manifold endowed with $\mathbb{Z}_2$--action \cite[Corollary 1.9]{DJ91}. The latter gives rise to a $\mathbb{Z}^n_2$--colouring that corresponds to a manifold known as a small cover. The technique of colourings of right--angled polyhedra was then generalised in the works \cite{KMT, KS?}, and the recent work of Choi and Park \cite{CP2} provides an efficient method to compute cohomology rings of manifolds obtained via colourings. Right--angled polyhedra attracted a lot of interest per se and as constituent pieces of hypebolic manifolds, cf. \cite{Atkinson, Dufour, Kolpakov, KM, PV}.   

Colourings of polytopes provide an accessible way to build hyperbolic manifolds with prescribed topological characteristics. Cf. the recent surveys \cite{Martelli, Vesnin} for more information about the recent results in this area. 

One of the classes of manifolds with relatively simple topological structure are homology spheres. We are mostly interested in \textit{rational homology spheres} (termed $QHS$ for short) and, moreover, in those that are hyperbolic manifolds (which we call hyperbolic $QHS$).  

Having a hyperbolic structure, however, implies that their fundamental groups are highly non--trivial, and this leaves only a narrow escape for hyperbolic $QHS$ to exist. For example, there are no hyperbolic QHS in dimension $2$. The first hyperbolic $QHS$ in dimension $3$ was constructed by Seifert and Weber \cite{SW}. A $QHS$ produced by a $\mathbb{Z}^3_2$--colouring of the right--angled hyperbolic dodecahedron can be found in the work of Garrison and Scott \cite{GS}. There is still no known example of a hyperbolic $QHS$ in dimensions $\geq 4$. Moreover, it is known that even such a large class as arithmetic hyperbolic manifolds cannot produce a QHS of even dimension $n > 4$ \cite{emery-qhs}. 



We can obtain some information about the so--called coloured symmetries \cite[Section 2.1]{KS?} of QHS obtained from $\mathbb{Z}^k_2$ ($k\geq 3$) colourings of compact right--angled polytopes in $\mathbb{H}^3$. Informally speaking, coloured symmetries correspond to self--isometries of such manifolds that can be read off the combinatorial structure of the colouring. Modulo the deck transformations coming from $\mathbb{Z}^k_2$, we are left with the so--called coloured symmetries of the polyhedron (endowed with a fixed colouring). 

\begin{thm}\label{thm:coloured-symmetries}
Let $\lambda$ be a rank $k$ orientable colouring of the right-angled, compact $\mathcal{P} \subset \mathbb{H}^3$ such that the corresponding manifold $\mathcal{M}_\lambda$ is a rational homology sphere. Then the admissible symmetry group $\mathrm{Sym}_{\lambda}(\mathcal{P})$ is isomorphic to a subgroup of $GL_k^{\mathrm{or}}(\mathbb{Z}_2)$. In particular, $\mathrm{Sym}_{\lambda}(\mathcal{P})\hookrightarrow S_4$ for $k=3$ and $\mathrm{Sym}_{\lambda}(\mathcal{P})\hookrightarrow\mathbb{Z}_2^3 \rtimes SL_3(\mathbb{Z}_2)$ for $k=4$. Moreover, $|\mathrm{Sym}_{\lambda}(\mathcal{P})|$ is odd for any rank $k>4$ colouring.
\end{thm}

We also obtained a full classification of the possible coloured symmetries of $QHS^3$ colourings, dividing the coloured symmetries into \text{good} or \text{bad}. A symmetry $\varphi$ is said to be good if the colours around $\mathrm{Fix}(\varphi|_{\partial}$ do not span the colouring space; and bad otherwise.

\begin{thm}\label{thm:possible symmetries and ranks}
Let $\lambda$ be a rank $k$ orientable colouring of $\mathcal{P} \subset \mathbb{H}^3$ such that $\mathcal{M}_\lambda$ is a $QHS$. Then $\varphi \in \mathrm{Sym}_{\lambda}(\mathcal{P})$ only if $\varphi$ is one of the following:
\begin{enumerate}
    \item a bad edge rotation (for $k\le 4$);
    \item a bad face-edge rotation (for $k=3$);
    \item a good edge rotoreflection (for $k\le 4$);
    \item a good odd-order face rotation;
    \item a vertex rotation;
    \item a face-vertex rotation.
\end{enumerate}
In particular, $\mathrm{Sym}_{\lambda}(\mathcal{P})$ does not contain any good order $2$ symmetry, and $\varphi \in \mathrm{Sym}_{\lambda}(\mathcal{P})$ is good only if $\varphi$ is an edge rotoreflection or an odd-order rotation.
\end{thm}

\subsection*{Outline of the paper}

In Section~\ref{sec:preliminaries} we provide all the necessary definition and basic facts about (hyperbolic) manifolds obtained via colourings of right--angled polyhedra, including results about their cohomology and coloured isometries. 

In Section~\ref{sec:3-qhs} we study possible obstructions to admissible symmetries of hyperbolic rational homology $3$--spheres and the general group structure of their coloured isometries. Here we prove Theorem~\ref{thm:coloured-symmetries} (which is Theorem~\ref{thm:admissible groups} and Corollary~\ref{cor:odd coloured symmetry groups} combined) and Theorem~\ref{thm:possible symmetries and ranks} (Theorem~\ref{thm:possible symmetries} and Proposition~\ref{prop:is good}). We also provide some illustrative examples. 

In Section~\ref{sec:constructions} we explain how to construct colourings that contain a given symmetry of the (uncoloured) polyhedron in their admissible symmetry groups.

In Section~\ref{sec:questions} some open questions are stated regarding colourings of right--angled polytopes and their admissible symmetries. 


\section{Preliminaries}\label{sec:preliminaries}

A finite-volume polytope $\mathcal{P} \subset \mathbb{X}^n$ (for $\mathbb{X}^n = \mathbb{S}^n, \mathbb{E}^n, \mathbb{H}^n$ being spherical, Euclidean and hyperbolic $n$--dimensional space, respectively, cf. \cite[Chapters 1--3]{Ratcliffe}) is called {\itshape right-angled} if any two codimension $1$ faces (or facets, for short) are either intersecting at right angles or disjoint. It is known that compact, hyperbolic right-angled polytopes cannot exist if $n > 4$. The only compact, right-angled spherical and Euclidean polytopes are the $n$-simplex and the $n$-parallelotope, while \cite[Theorem 2.4]{Vesnin} gives us a sufficient condition for abstract $3$-polytopes to be realizable as right-angled, hyperbolic polytopes. There is no such classification for right-angled $n$-polytopes with $n\geq 4$.


\subsection{Colourings of right--angled polytopes}
One of the important properties of hyperbolic right-angled polytopes is that their so-called colourings provide a rich class of hyperbolic manifolds. By inspecting the combinatorics of a colouring, one may obtain important topological and geometric information about the associated manifold.

\begin{defn}
Let $\mathcal{S}$ be an $n$--dimensional simplex with the set of vertices $\mathcal{V}=\{v_1,\dots,v_{n+1}\}$ and $W$ an $\mathbb{Z}_2$--vector space. A map $\lambda: \mathcal{V} \rightarrow W$ is called a {\itshape colouring} of $\mathcal{S}$. The colouring of $\mathcal{S}$ is {\itshape proper} if the vectors $\lambda(v_i), \ i=1,\dots, n+1$ are linearly independent. 

Let $K$ be a simplicial complex with the set of vertices $\mathcal{V}$. Then a {\itshape colouring} of $K$ is a map $\lambda: \mathcal{V} \rightarrow W$. A colouring $\lambda$ of $K$ is {\itshape proper} if $\lambda$ is proper on each simplex in $K$. 
\end{defn}

Notice that if $\mathcal{P} \subset \mathbb{X}^n$ is a compact right-angled polytope then $\mathcal{P}$ is necessarily simple and the boundary complex $K_\mathcal{P}$ of its dual $\mathcal{P}^*$ is a simplicial complex.

\begin{defn}
Let $\mathcal{P} \subset \mathbb{X}^n$ be a polytope with the set of facets $\mathcal{F}$ and $K_\mathcal{P}$ be the maximal simplicial subcomplex of the boundary of $\mathcal{P}^*$. A \textit{colouring} of $\mathcal{P}$ is a map $\lambda: \mathcal{F} \rightarrow W$, where $W$ an $\mathbb{Z}_2$--vector space. This map naturally defines a colouring of $K_\mathcal{P}$. Then $\lambda$ is called \textit{proper} if the induced colouring on $K_\mathcal{P}$ is proper. 
\end{defn}

If the polytope $\mathcal{P}$ or the vector space $W$ are clear from the context, then we will omit them and simply refer to $\lambda$ as a colouring. The \textit{rank} of $\lambda$ is the $\mathbb{Z}_2$--dimension of $\mathrm{im}\, \lambda$. We will always assume that colourings are surjective, in the sense that the image of the map $\lambda$ is a generating set of vectors for $W$.

A colouring of a right-angled $n$-polytope $\mathcal{P}$ naturally defines a homomorphism, which we still denote by $\lambda$ without much ambiguity, from the associated right-angled Coxeter group~$\Gamma(\mathcal{P})$ (generated by reflections in all the facets of $\mathcal{P}$) into $W$ (with its natural group structure). Being a Coxeter polytope, $\mathcal{P}$ has a natural orbifold structure as the quotient $\mathbb{X}^n /_{\Gamma(\mathcal{P)}}$.

\begin{prop}[\cite{DJ91}, Proposition 1.7]\label{DJ}
If the colouring $\lambda$ is proper, then $\ker \lambda < \Gamma(\mathcal{P})$ is torsion-free, and $\mathcal{M}_\lambda = \mathbb{X}^n /_{\ker \lambda}$ is a closed manifold.
\end{prop}

If the dimension of the vector space $W$ is minimal, i.e. equal to the maximum number of vertices in the simplices of $\mathcal{P}^*$, $\mathcal{M}_\lambda$ is called a \textit{small cover} of $\mathcal{P}$. This is equivalent to $\text{rnk} \, \lambda =n$.

\begin{rem}\label{rem:manifold cover tesselation}
If $\lambda:\Gamma \to \mathbb{Z}_2 ^m$ is surjective, $M_\lambda$ is homeomorphic to the topological manifold $(\mathcal{P} \times \mathbb{Z}_2 ^m)/_{\sim}$, where
\begin{equation}\label{colouring 2}
    (p,g) \sim (q,h) \Leftrightarrow p=q \text{ and } g - h \in W_p,
\end{equation}
with $W_p = \lambda(\mathrm{Stab}_\Gamma(f))$ being the subspace of $W = \mathbb{Z}_2 ^k$ generated by $\lambda(F_{i_1})$, \ldots, $\lambda(F_{i_l})$ for $f= F_{i_1} \cap \ldots \cap F_{i_l}$ such that $p\in \mathrm{int}(f)$. If $p\in\mathrm{int}(\mathcal{P})$, we put $W_p=\{0\}$.

More practically, this means that $(\mathcal{P} \times \mathbb{Z}_2 ^k)/_{\sim}$ can be obtained from $\mathcal{P} \times \mathbb{Z}_2 ^k$ by identifying distinct copies $F\times \{g\}$ and $F\times \{h\}$ through the reflection on $F \in \mathcal{F}$ (which acts on the orbifold $\mathcal{P}$ as the identity) whenever $g-h =\lambda(F)$. It is clear then that $M_\lambda$ can be tessellated by $2^k$ copies of $\mathcal{P}$.

\end{rem}

\begin{defn}
Let $\mathrm{Sym}(\mathcal{P})$ be the group of symmetries of a polytope $\mathcal{P}\subset \mathbb{X}^n$, which acts on the set of facets $\mathcal{F}$ as a permutation. Two $W$--colourings $\lambda,\mu:\mathcal{F}\to W$ of $\mathcal{P}$ are called \textit{equivalent} if there exists a combinatorial symmetry $s \in \mathrm{Sym}(\mathcal{P})$, acting as a permutation on $\mathcal{F}$, and an invertible linear transformation $m \in \mathrm{GL}(W)$ such that $\lambda = m \circ \mu \circ s$. 
\end{defn}

It is easy to see that $DJ$-equivalent proper colourings of a polytope $\mathcal{P} \subset \mathbb{X}^n$ define isometric manifolds. The converse is also known to be true for orientable manifold covers of the $3$-cube \cite[Proposition 3.2]{FKS} and orientable small covers of compact, hyperbolic $3$-polytopes \cite[Theorem 3.13]{Vesnin}.


\subsection{Orientable colourings and orientable extensions}

We say that a (not necessarily proper) $\mathbb{Z}_2^k$--colouring $\lambda$ is \textit{orientable} if the orbifold $M_\lambda$ is. We have the following criterion to determine if a colouring is orientable:

\begin{prop}[\cite{KMT}, Lemma 2.4]\label{orientability}
$M_\lambda$ is orientable if and only if $\lambda$ is equivalent to a colouring that assigns to each facet a colour with an odd number of entries 1.
\end{prop}





Thus, when looking for orientable colourings, up to equivalence, we can restrict our attention to colours with an odd number of entries 1, that is, the elements of
$$\mathbb{Z}_2^{k,\mathrm{or}}=\left\{v \in \mathbb{Z}_2^k \Bigr\vert \sum_{i=1}^k v_i = 1 \right\}.$$ We say that a colour is \textit{orientable} if it's an element of some $\mathbb{Z}_2^{k,\mathrm{or}}$, and \textit{non-orientable} otherwise.

If we restrict ourselves to use only orientable colours, checking for properness can be done more easily than in the general setting. Indeed, we have the following:

\begin{lemma}
An orientable colouring is proper if and only if for all $\sigma=\{i_0, \ldots, i_s\} \in K$ simplices with $s$ odd, $\lambda(i_0)+ \ldots +\lambda(i_s)\neq 0$. 
\end{lemma}

\begin{proof}
If $\lambda(i_0)+ \ldots +\lambda(i_s)\neq 0$ for every simplex $\sigma=\{i_0, \ldots ,i_s\} \in K$, then the colouring is proper. But this sums are always non-zero for $s$ even, since any sum of an odd number or $\mathbb{Z}_2$--vectors with an odd number of entries 1 cannot be null. It suffices then to request that the sums are non-zero for $s$ odd.
\end{proof}

In dimension 3, this means that an orientable colouring is proper if and only if adjacent faces have different colours, due to the fact that we don't need to check for properness on the triangles dual to vertices of the polytope. We can then forget about the triangles and consider only the $1$-skeleton of $K_\mathcal{P}$. This is a graph, and it will be denoted by $G(\mathcal{P})$.

As such, any $2^{k-1}$-colouring on the vertices of $G(\mathcal{P})$ together with a bijection $\eta$ sending the vertex colours (numbers from $1$ to $2^{k-1}$) into distinct vectors of $\mathbb{Z}_2^{k,\mathrm{or}}$ will give us a proper orientable rank $k$ colouring of $\mathcal{P}$. 

The colouring equivalence under the polytopes, however, is finer than the one in the graphs. Indeed, let $$GL^\mathrm{or}_k(\mathbb{Z}_2)= \{ g \in GL_k(\mathbb{Z}_2) \mid g(\mathbb{Z}_2^{k,\mathrm{or}})=\mathbb{Z}_2^{k,\mathrm{or}}\} $$
be the group of isomorphisms of $\mathbb{Z}_2^k$ that preserves orientable colours (that is, matrices with orientable vectors as columns). Generally, the permutation of colours induced by $GL^\mathrm{or}_k(\mathbb{Z}_2)$ through $\eta^{-1}$ into $S_{2^{k-1}}$ might not deliver the whole permutation group. On small covers, whoever, $GL^\mathrm{or}_3(\mathbb{Z}_2)\cong S_4$ and we have the following equivalence:

$$ \faktor{\{\, \text{orientable small covers of} \, \mathcal{P} \subset \mathbb{H}^3\}}{\sim} \Leftrightarrow  \faktor{\{\, 4\text{-colourings of} \, G(\mathcal{P})\}}{\sim}.$$


\subsection{Colouring Extensions}\label{subsection:colouring extensions}

We briefly recall the following definition from \cite[Section 7.2]{FKR}:

\begin{defn}
Let $\lambda:\mathcal{F}\to \mathbb{Z}_2 ^k$ be any colouring. A (surjective) colouring $\mu:\mathcal{F}\to \mathbb{Z}_2 ^{k+1}$ is called \textit{an extension of} $\lambda$ if there is a linear projection $p:\mathbb{Z}_2 ^{k+1}\to \mathbb{Z}_2 ^k$ such that $\lambda=p\circ\mu$.
\end{defn}

\cite[Proposition 7.2]{FKR} gives us the following:

\begin{prop}\label{prop:double-cover colouring}
Let $\lambda:\mathcal{F}\to \mathbb{Z}_2 ^k$ be any colouring and $\mu:\mathcal{F}\to \mathbb{Z}_2 ^{k+1}$ its extension. Then $\mathcal{M}_\mu$ double-covers $\mathcal{M}_\lambda$. Moreover, if $\lambda$ is proper or orientable, so is $\mu$.
\end{prop}


\subsection{Symmetries of colourings}\label{subsection:symmetries}
Given a (not necessarily proper) colouring $\lambda$ of a right-angled polytope $\mathcal{P}\subset\mathbb{X}^n$, there is a natural group of isometries of the associated orbifold $\mathcal{M}_{\lambda}$ called its \textit{coloured isometry group}. We also recall its definition from \cite[Section 2.1]{KS?}.

\begin{defn}\label{def:admissible-symmetries}
Let $\lambda$ be a $W$--colouring of $\mathcal{P}$. A symmetry $\varphi$ of $\mathcal{P}$ is \textit{admissible} with respect to $\lambda$ if:
\begin{enumerate}
    \item the maps $\varphi$ induces a permutation of the colours assigned by $\lambda$ to the facets of $\mathcal{P}$,
    \item such permutation is realised by an invertible linear automorphism in $\mathrm{GL}(W)$. 
    \end{enumerate}
\end{defn}

Admissible symmetries are easily seen to form a subgroup of the symmetry group of $\mathcal{P}$, called \textit{coloured symmetry group}, which we denote by $\mathrm{Sym}_{\lambda}(\mathcal{P})$,  and there is a naturally defined homomorphism $\Psi: \mathrm{Sym}_{\lambda}(\mathcal{P}) \rightarrow \mathrm{GL}(W)$.

Recall that a colouring $\lambda:\Gamma(\mathcal{P})\rightarrow W$  defines a regular orbifold cover $\pi: \mathcal{M}_{\lambda}\rightarrow \mathcal{P}$ with automorphism group $W$. The coloured isometry group $\mathrm{Isom}_c(\mathcal{M}_{\lambda})$ is defined as the group of symmetries of $\mathcal{M}_{\lambda}$ which are lifts of admissible symmetries of $\mathcal{P}$. We have that  
\begin{equation}\label{eq:semidirect_product}
    \mathrm{Isom}_c(\mathcal{M}_{\lambda})\cong \mathrm{Sym}_{\lambda}(\mathcal{P})\ltimes W,
\end{equation}
where the action of $\mathrm{Sym}_{\lambda}(\mathcal{P})$ on $W$ is precisely the one induced by the homomorphism $\Psi$ \cite[Section 2.1]{KS?}. 


\subsection{Computing the homology of colourings.} Given a right-angled polytope $\mathcal{P} \subset \mathbb{X}^n$ with an $\mathbb{Z}^k_2$--colouring $\lambda$, let us enumerate the facets $\mathcal{F}$ of $\mathcal{P}$ in some order. Then, we may think that $\mathcal{F} = \{ 1, 2, \ldots, m \}$. Let $\Lambda$ be the \textit{defining matrix} of $\lambda$, that consists of the column vectors $\lambda(1), \ldots, \lambda(m)$ exactly in this order. Then $\Lambda$ is a matrix with $k$ rows and $m$ columns. More precisely, $\Lambda$ is the abelianization of $\lambda$, that is, the map such that $\Lambda \circ \mathrm{ab}=\lambda$, where $\mathrm{ab}:\Gamma \to \mathbb{Z}_2 ^m$ is the abelianization map that takes $r_i$, the reflection of the facet $i$, to $e_i$. 

Let $\mathrm{Row}(\Lambda)$ denote the row space of $\Lambda$, while for a vector $\omega \in \mathrm{Row}(\Lambda)$ let $K_\omega$ be the simplicial subcomplex of the complex $K=K_\mathcal{P}$ spanned by the vertices $i$ (also labelled by the elements of $\{1, 2, \ldots, m\}$) such that the $i$--th entry of $\omega$ equals $1$.

The rational cohomology of $\mathcal{M}_{\lambda}$ can be computed through the following formula, due to  \cite[Theorem 1.1]{CP2} by taking $R=\mathbb{Q}$:

\begin{equation}\label{eq:cohomology}
H^p(\mathcal{M}_{\lambda},\mathbb{Q})\cong \underset{\omega \in \mathrm{Row}(\Lambda)}{\bigoplus}\widetilde{H}^{p-1}(K_{\omega},\mathbb{Q}). 
\end{equation}

Moreover, the cup product structure is given by the maps \cite[Main Theorem]{CP2}:

\begin{equation}\label{eq:cup products}
\widetilde{H}^{p-1}(K_{\omega_1},\mathbb{Q}) \otimes \widetilde{H}^{q-1}(K_{\omega_2},\mathbb{Q}) \mapsto \widetilde{H}^{p+q-1}(K_{\omega_1 + \omega_2},\mathbb{Q}). 
\end{equation}

One practical application of Equation \ref{eq:cohomology} is the computation of Betti numbers of colourings of $n$-cubes, which turns out to be very simple.

\begin{cor}[\cite{Ferrari}, Proposition 5.1.11]\label{cor:Ferrari}
Let $\mathcal{C}$ be an $n$-cube with opposing facets labelled with consecutive numbers $\{2i-1,2i\}$ for $i \in [n]$, and $\lambda$ be any colouring of $\mathcal{C}$. Finally, let $T_j\subset \mathbb{Z}_2^{2n}$, $0\le j \le n$, be the sets such that
\begin{align*}
    T_1 =&\{e_1+e_2, \ldots, e_{2i-1}+e_{2i}, \ldots, e_{2n-1}+e_{2n}\} \\
    T_j =&\big\{x_1+\ldots+x_j \,\big|\, \{x_1, \ldots, x_j\} \subset T_1\big\}.
\end{align*}
For any $j \in \{1, \ldots, n\}$, we have the following:
$$\beta_j ^\mathbb{Q} (\mathcal{M}_\lambda)=\big|\mathrm{Row} (\Lambda) \cap T_j\big|. $$
\end{cor}

We conclude this section by proving the following useful proposition.

\begin{cor}\label{prop:orientability_criterion}
Let $\mathcal{P} \subset \mathbb{X}^n$ be a compact, right-angled polytope with a proper colouring $\lambda$. The manifold $\mathcal{M}_{\lambda}$ is orientable if and only if the vector $\varepsilon=(1,\dots,1)$ belongs to $\mathrm{Row}(\Lambda)$. 
\end{cor}

\begin{proof}
By Proposition \ref{orientability}, if $\lambda$ is orientable, then the sum of all lines of $\Lambda$ is the vector $\varepsilon$. Conversely, if $\varepsilon \in \mathrm{Row}(\Lambda)$, then $K_{\varepsilon}=K\sim \mathbb{S}^{n-1}$ and $\widetilde{H}^{n-1}(K_{\varepsilon},\mathbb{Q})\cong \mathbb{Q} \Rightarrow H^n(\mathcal{M}_{\lambda},\mathbb{Q})\cong \mathbb{Q}$ by Equation \ref{eq:cohomology}. We thus must have that $\mathcal{M}_{\lambda}$ is orientable.
\end{proof}


\section{Rational Homology 3-Spheres}\label{sec:3-qhs}

A manifold $X$ is said to be a \textit{homology $n$-sphere} (or $HS^n$) if $H_i(X, \mathbb{Z})=H_i(\mathbb{S}^n, \mathbb{Z}), \forall i$. Similarly, it is said to be a \textit{ rational homology $n$-sphere} (or $QHS^n$) if $H_i(X, \mathbb{Q})=H_i(\mathbb{S}^n, \mathbb{Q}), \forall i$. Clearly, any $HS$ or $QHS$ manifold must be orientable. Additionally, we say a $CW$-complex is a \textit{rational homology point} (or $QHP$) if all its reduced rational homology groups are trivial.

\subsection{Subcomplexes of $QHS^3$}

By an application of Equation \ref{eq:cohomology}, we have the following:

\begin{lemma}\label{lemma:subcomplexes are HP}
An orientable $\mathcal{M}_{\lambda}$ is a $QHS^n$ if and only if, for all $K_\omega \in \mathrm{Row}(\Lambda) \setminus \{0, \varepsilon\}$, $K_\omega$ is an $QHP$.
\end{lemma}

\begin{proof}
The only non-trivial cohomology groups of $\mathcal{M}_{\lambda}$ are $H^0(\mathcal{M}_{\lambda},\mathbb{Q})\cong \widetilde{H}^{-1}(K_0,\mathbb{Q})\cong \mathbb{Q}$ and $H^n(\mathcal{M}_{\lambda},\mathbb{Q})\cong \widetilde{H}^{n-1}(K_\varepsilon,\mathbb{Q})\cong \mathbb{Q}$. Therefore, every other simplicial subcomplex $K_\omega$ must have trivial reduced homology groups.
\end{proof}

In the case of colourings of $n$-cubes, by applying Corollary \ref{cor:Ferrari}, we have that:

\begin{prop}
Let $\mathcal{C}$ be an $n$-cube labelled as in Corollary \ref{cor:Ferrari} and $T_i$ its $T$-sets. Put $\mathcal{T}=\bigcup_{i<n} T_i$. We have that $\mathcal{M}_\lambda$ is a $QHS$ if and only if $\mathrm{Row}\Lambda \cap \mathcal{T}=\emptyset$. 
\end{prop}

By applying Equation \ref{eq:cup products}, by the other hand, we have the following:

\begin{lemma}\label{lemma:HP condition}
Let $\mathcal{M}_{\lambda}$ be an orientable colouring and $\omega \in \mathrm{Row}(\Lambda)\setminus\{0,\epsilon\}$. Then $K_\omega$ is an $QHP$ if and only if $K_{\epsilon-\omega}$ is an $QHP$.
\end{lemma}

\begin{proof}
Assume $K_\omega$ is not an $HP$. Then $\tilde{H}^*(K_\omega,\mathbb{Q})$ has either $0$-cohomology or $1$-cohomology. Assume $0\neq \alpha \in \widetilde{H}^0(K_{\omega},\mathbb{Q})$. By Equation \ref{eq:cohomology}, $\alpha \in H^1(\mathcal{M}_{\lambda},\mathbb{Q})$. By \cite[Corollary 3.39]{Hatcher}, there exists $\beta \in H^2(\mathcal{M}_{\lambda},\mathbb{Q})$ such that $\alpha\smile \beta$ is the generator of $H^3(\mathcal{M}_{\lambda},\mathbb{Q})$. By Equation \ref{eq:cohomology}, then, $\alpha \smile \beta$ is the generator of $\widetilde{H}^2(K,\mathbb{Q})$, since $K_\epsilon=K$ and $|K|\sim \mathbb{S}^2$. Finally, by Equations \ref{eq:cohomology} and \ref{eq:cup products}, we must have that $0\neq \beta \in \widetilde{H}^1(K_{\epsilon-\omega},\mathbb{Q})$, otherwise the product $\alpha \smile \beta$ wouldn't lie in $\widetilde{H}^2(K,\mathbb{Q})$. The case with $\widetilde{H}^1(K_{\omega},\mathbb{Q})\neq 0$ is analog.
\end{proof}

Algorithmically, thus, we can improve Lemma \ref{lemma:subcomplexes are HP} by checking only connectedness of some graphs.

\begin{cor}\label{cor:algorithm}
An orientable $\mathcal{M}_{\lambda}$ is a $QHS^3$ if and only if, for all $\omega \in \mathrm{Row}(\Lambda) \setminus \{0, \varepsilon\}$, the 1-skeleton of $K_\omega$ is connected.
\end{cor}

\begin{proof}
If one proper subcomplex $K_\omega$ has a non-trivial cycle, by Lemma \ref{lemma:HP condition}, the complementary complex $K_{\varepsilon-\omega}$ will be disconnected. It suffices therefore to check if all proper, non-empty subcomplexes $K_\omega$ are connected. But since connectedness depends only on the $1$-skeleton of $K_\omega$, it is enough to check for connectedness on those.
\end{proof}

In the small cover case, the proper subcomplexes have a even better description, and the algorithm can be greatly improved. Given a coloured graph $G$, let $G_X$ be the subgraph containing all the vertices coloured with $x\in X$.

\begin{cor}\label{cor:small cover QHS subgraphs are trees}
Let $\mathcal{M}_{\lambda}$ be a orientable small cover and $G=G(\mathcal{P})$ its $4$-coloured associated graph. Then $\mathcal{M}_{\lambda}$ is a $QHS^3$ if and only if $G_{1,2},G_{1,3}$ and $G_{2,3}$ are trees. In that case, $$\big\{K_\omega \mid \omega \in \mathrm{Row}(\Lambda) \setminus \{0, \varepsilon\}\big\}=\{G_X\mid X\subset [4], |X|=2\}$$ and each of those subcomplexes is a tree.
\end{cor}

\begin{proof}
Let $\varphi$ be the bijection between colours that take $i$ to $e_i$ and $4$ to $e_1+e_2+e_3$. The complex $K_\omega$ corresponding to the line $i$ of $\Lambda$ is the maximal subcomplex containing all vertices of $K$ with colours $i$ and $4$. Similarly, the complex $K_\omega$ corresponding to the sum of the lines $i$ and $j$ in row $\Lambda$ is the maximal subcomplex containing all vertices of $K$ with colours $i$ and $j$. The six subcomplexes in $\big\{K_\omega \mid \omega \in \mathrm{Row}(\Lambda) \setminus \{0, \varepsilon\}\big\}$ therefore consist on the maximal subcomplexes of $K$ containing any combination of two colours in $\{1,2,3,4\}$. 

Since these subcomplexes have only two colours, they cannot contain any triangle, and thus are naturally graphs. And since they must be $HP$, they must be trees. We can look at them, therefore, as the subgraphs $G_{i,j}$. But, by Lemma \ref{lemma:HP condition}, $G_X$ is a tree if and only if $G_{[4]\setminus X}$ is a tree. It suffices thus to check if one subgraph out of each pair $\{G_X,G_{[4]\setminus X}\}$ is a tree.
\end{proof}

\begin{rem}
This algorithm for small covers is significantly faster than the general one proposed in Corollary \ref{cor:algorithm} if one takes into consideration that $QHS$ are rare. Indeed, finding sooner a subcomplex which provides homology to the colouring via Equation \ref{eq:cohomology} allows us to determine sooner when colourings aren't $QHS$. As such, although somewhat counter-intuitive, checking if $3$ subgraphs are trees ends up being faster than checking if all $6$ subcomplexes are connected. Looking directly for subgraphs on specific colours, moreover, allows us to avoid computing the row space of the colouring matrix, which experimentally would consume the greater part of computational time.
\end{rem}

We conclude with a number of lemmas which will be useful for the analysis of possible coloured isometries of $QHS^3$ manifold covers.

\begin{lemma}\label{lemma:parity of complexes}
Let $\Lambda$ be a rank $k$ colouring of $\mathcal{P} \subset \mathbb{X}^n$. Then the number of subcomplexes in $\big\{K_\omega \mid \omega \in \mathrm{Row}(\Lambda) \setminus \{0, \varepsilon\}\big\}$ containing a set of colours of rank $s<k$ is $2^{k-s}-1$. In particular, it is always odd.
\end{lemma}

\begin{proof}
Choose a basis to the set of colours. Up to equivalence, these vectors can be chosen to be $e_1, \ldots, e_s$. Thus, the subcomplexes $K_\omega$ containing all vertices with those colours are precisely the sums of lines of in $\Lambda$ that contain the first $s$ lines. Since $\Lambda$ has exactly $k$ lines and we must exclude the sum of all lines, which yields $\epsilon$, the number of subcomplexes is precisely $2^{k-s}-1$.
\end{proof}

\begin{lemma}\label{lemma:fixed point}
Let $X$ be an $HP$ simplicial complex and $f:|X|\to |X|$ be a continuous map. Then $f$ has a fixed point.
\end{lemma}

\begin{proof}
The Lefschetz number of $f$ is $$L(f)=\sum_{i\ge 0} (-1)^i \mathrm{tr} \, H_i(f,\mathbb{Q}) \big)=\mathrm{tr}\, H_0(f,\mathbb{Q}),$$ where $H_0(f,\mathbb{Q})$ is precisely the matrix of the permutation of connected components of $X$ induced by $f$. Since $X$ is connected, $H_0(f,\mathbb{Q})=[1]$. We conclude that $L(f)\neq 0$ and $f$ has a fixed point by the Lefschetz fixed point theorem.
\end{proof}


\subsection{Topological obstructions to admissible symmetries}

In the previous subsection, we described how the subcomplexes of $K=K_\mathcal{P}$ give us information about the cohomology of $\mathcal{M}_\lambda$. As it turns out, these subcomplexes can also give us a lot of information about the group $\mathrm{Sym}_\lambda(\mathcal{P})$, specially when $\mathcal{M}_\lambda$ is a $QHS^3$. We'll denote the set of proper subcomplexes by $K_\Omega = \big\{K_\omega \mid \omega \in \mathrm{Row}(\Lambda) \setminus \{0, \varepsilon\}\big\}$. Also, for convenience of notation, we will denote $\mathrm{Fix}(\varphi|_{\partial \mathcal{P}})$ by $\mathrm{Fix}(\varphi)$ throughout the whole section, unless stated otherwise.

Indeed, $\mathrm{Sym}_\lambda(\mathcal{P})$ acts naturally on $\mathbb{Z}_2^{k,\mathrm{or}}$ as a permutation of colours, and this induces a permutation on $K_\Omega$, where $\mathrm{Sym}_\lambda(\mathcal{P})$ acts as a simplicial isomorphism between subcomplexes. By analising the order of the induced permutation of an admissible symmetry $\varphi$ on these sets, we can determine whether the symmetry fixes a proper subcomplex $K_\omega$ or not. We note that, by Corollary \ref{prop:orientability_criterion}, $\epsilon \in \,\mathrm{Row}\,\Lambda$ and $|K_\Omega|=2^k -2$.

We can already deduce an obstruction to antipodal symmetries from this.

\begin{prop}\label{antipodal symmetries}
Let $\lambda$ be an orientable colouring of $\mathcal{P} \subset \mathbb{X}^n$ such that $\mathcal{M}_\lambda$ is a $QHS$. Then $\mathrm{Sym}_{\lambda}(\mathcal{P})$ does not contain antipodal symmetries.
\end{prop}

\begin{proof}
If $\mathcal{P}$ has no antipodal symmetry, we are done. Otherwise, let $\alpha \in \mathrm{Sym}(\mathcal{P})$ be the antipodal symmetry. Assume $\alpha$ is admissible with respect to $\lambda$. Take any vertex $v$ of $K=(\partial \mathcal{P})^*$ and its image $a(v)$. The symmetry $a$ fixes the set $\{v,a(v)\}$, and thus fixes the set of colours $\{\lambda(v),\lambda\big(a(v)\big)\}$ (which could contain only one element). By Lemma \ref{lemma:parity of complexes}, the number of complexes in $K_\Omega$ that contain this set of colours is odd. 

Since $a$ acts as a permutation on this set and has order $2$, it must fix one element $K_\omega$ of that set.
However, due to the fact that $K_\omega$ is a $HP$, Lemma \ref{lemma:fixed point} implies that $\alpha$ has a fixed point in $|K_\omega|$, an absurd since the antipodal map has no fixed points. Therefore $\alpha$ cannot be an admissible symmetry.
\end{proof}

We now analyse the other possible symmetries in $\mathrm{Sym}(\mathcal{P}) < O(2)$. The finite order elements of $O(2)$ are rotations, reflections and rotoreflections (a composition of a rotation and a reflection when they commute). Rotations fix the poles of their rotation axis, while reflections fix the meridian where the reflection plane intersects the sphere. Rotoreflections, however, have no fixed point. We note that rotoreflections must have even order, and that order $2$ rotoreflections is the antipodal map. Moreover, an order $2s$ rotoreflections for $s>1$ even is necessarily the composition of an order $2s$ rotation and a reflection, but an order $2s$ one for $s>1$ odd can be the composition of a reflection with an order $s$ rotation or an order $2s$ rotation.

If any rotation fix $\mathcal{P}$, it must be a face rotation, an edge rotation or vertex rotation, depending whether the axis of the rotation is the midpoint of a face, edge or vertex. But note that if a face is opposed by an vertex or an edge, then their rotation induces a rotation on the $n$-gon whenever $2|n$ or $3|n$, respectively. Rotoreflections, by the other hand, must have the same type of midpoint on their poles, although then don't need to be symmetrical by a reflection (the edge rotoreflection, for instance, combines a $\frac{\pi}{2}$ rotation with the reflection). Finally, we point out that that right-angled $3$-polytopes have vertex valence $3$, thus their vertex rotation must necessarily be of order $3$. We list the possible symmetries, their orders, whether they are orientation preserving or reversing, and their fixed points on the boundary of the polytope in Table \ref{tabela1}.

\medskip

\begin{defn}
An admissible symmetry is called \textit{good} if the colours around its fixed points (i.e. the colours in the set $\{\lambda(F) \mid F \in \mathcal{F}, F \cap \mathrm{Fix}(\varphi)\neq\emptyset\}$) do not span the whole colouring space.
\end{defn}

\begin{prop}\label{prop:is good}
Let $\mathcal{P} \subset \mathbb{X}^3$, $\lambda$ be a rank $k$ colouring and $\varphi \in \mathrm{Sym}_\lambda(\mathcal{P})$. Suppose one of the following is true:
\begin{enumerate}
    \item $\varphi$ is a rotoreflection;
    \item $\varphi$ is a face rotation;
    \item $\varphi$ is a face-edge rotation and $k\ge 4$;
    \item $\varphi$ is a edge rotation and $k\ge 5$;
    \item $\varphi$ is a face-vertex rotation and $k\ge 5$;
    \item $\varphi$ is a vertex rotation and $k\ge 7$.
\end{enumerate}
Then $\varphi$ is good.
\end{prop}

\begin{proof} 
Rotoreflections have no fixed points in $\partial \mathcal{P}$, therefore are always good. In any of the other cases, let $\mathrm{Fix}(\varphi)=\{p,q\}$ and consider the subset of faces $\mathcal{F}'=\{F \in \mathcal{F} \mid p \in F \,\text{or}\, q \in F\}$. By definition, if $k>|\mathcal{F}'|$, then $\varphi$ is good. We conclude by noting that the midpoint of a face in contained in a single face, the midpoint of an edge is contained in two faces and a vertex is contained in three faces.
\end{proof}

\begin{table}[h]
\centering
\caption{Possible symmetries of a right-angled $3$-polytope}
\label{tabela1} 
\begin{tabular}{ |c||ccc|  }
 \hline
 symmetry & order & orientation & fixed points in $\partial \mathcal{P}$ \\
 \hline
 \hline
 edge rotation & 2 & preserving & midpoints of opposing edges \\
 $n$-gon rotation & $s$ for $s|n$ & preserving & midpoints of faces \\
 vertex rotation & 3 & preserving & opposing vertices \\
 face-edge rotation & 2 & preserving & midpoints of edge and $2n$-gon \\
 face-vertex rotations & 3 & preserving & vertex and midpoint of $3n$-gon\\
 reflection & 2 & reversing & meridian of $\mathbb{S}^2$ \\
 antipodal & 2 & reversing & none \\
 edge rotoreflection & 4 & reversing & none \\
 vertex rotoreflection & 6 & reversing & none \\
 $n$-gon rotoreflection & $2s$ for $s|n$ odd & reversing & none \\
   & $s$ or $2s$ for $s|n$ even & reversing & none \\
 \hline
\end{tabular}
\end{table}


We also have the following observation.

\begin{prop}\label{prop: no good order 2}
Let $\lambda$ be an orientable colouring of $\mathcal{P} \subset \mathbb{X}^3$ such that $\mathcal{M}_\lambda$ is a $QHS$. Then $\mathrm{Sym}_{\lambda}(\mathcal{P})$ does not contain good order $2$ rotations or reflections.
\end{prop}

\begin{proof}
Suppose $\varphi \in \mathrm{Sym}_{\lambda}(\mathcal{P})$ is a good order $2$ rotation or reflection. By definition, there is at least one proper subcomplex in $K_\Omega$ that contains $\mathrm{Fix}(\varphi)$. Moreover, by Lemma \ref{lemma:parity of complexes}, the number of subcomplexes in $K_\Omega$ containing the colours around $\mathrm{Fix}(\varphi)$ (and therefore containing $\mathrm{Fix}(\varphi)$ in their interior), is odd. It follows that $\varphi$ fixes some proper subcomplex $K_\omega \in K_\Omega$, and thus it also fixed its complementary subcomplex $K_{\epsilon - \omega}$. But $K_{\epsilon-\omega}$ is a $QHP$, and we have that $\varphi$ must fix a point in $|K_{\epsilon-\omega}|$ by Lemma \ref{lemma:fixed point}, an absurd since $|K_{\epsilon-\omega}| \cap \, \mathrm{Fix}(\varphi) = \emptyset$.
\end{proof}

\begin{prop}\label{prop: no reflections}
Let $\lambda$ be an orientable colouring of $\mathcal{P} \subset \mathbb{H}^3$ such that $\mathcal{M}_\lambda$ is a $QHS$. Then $\mathrm{Sym}_{\lambda}(\mathcal{P})$ does not contain reflections.
\end{prop}

\begin{proof}
Suppose that there exists $\varphi \in \mathrm{Sym}_{\lambda}(\mathcal{P})$ reflection. By Proposition \ref{prop: no good order 2}, $\varphi$ is not good.

Assume $\mathrm{Fix}(\varphi)\cap \mathcal{P}$ does not contain edges of $\mathcal{P}$. Then $P=\mathrm{Fix}(\varphi)\cap \mathcal{P}$ is a right-angled $n$-gon. Moreover, by \cite[Theorem 2.4]{Vesnin}, $n\ge 5$. Now, consider the colouring $\mu$ induced on $P$ by $\lambda$. Since $P$ is a hyperbolic, right-angled $n$-gon, by Proposition \ref{DJ}, $\mathcal{M}_\mu \cong S_g$, with $g\ge2$. In particular, $\mathcal{M}_\mu$ is not a $QHS$. Then, by Lemma \ref{lemma:subcomplexes are HP}, there is a $\omega_0 \in \mathrm{Row} \Lambda \setminus \{0,\epsilon\}$ such that $(\partial P)^*_{\omega_0}=K_{\omega_0} \cap \mathrm{Fix}(\varphi)$ is disconnected.

Now, since $\varphi$ is not good, $\big\{\lambda(F) \mid \varphi(F)=F \big\}$ spans the colouring space. And since its induced action must fix the colours in this set, it fixes all the colours in $\mathbb{Z}_2^{k,\mathrm{or}}$ and thus all proper subcomplexes $K_\omega \in K_\Omega$. In particular, it fixes $K_{\omega_0}$. 

As such, take two points $p,q$ in distinct connected components of $K_{\omega_0}\cap \mathrm{Fix}(\varphi)$. By Lemma \ref{lemma:subcomplexes are HP}, there is a path $l$ connecting $p$ to $q$ in $K_{\omega_0}$. Since $\varphi$ fixes $\partial l$ pointwise, we have that $\varphi(l)\neq l$, otherwise $\mathrm{int}(l) \subset K_{\omega_0}\cap \mathrm{Fix}(\varphi)$, an absurd. It follows that $l\cup \varphi(l)$ is a cycle in $K_{\omega_0}$. By Lemma \ref{lemma:subcomplexes are HP}, there must be a disc $D \subset K_\omega$ (or a wedge of discs if $l$ and $\varphi(l)$ intersect in the interior) such that $\partial D = l\cup\varphi(l)$. If $D$ is fixed under $\varphi$, we can find a path in $\mathrm{Fix}(\varphi) \cap D$ from $p$ to $q$, an absurd. Otherwise, $D\cup \varphi(D)$ is a non-trivial $2$-chain, also an absurd.

Now, assume $\mathrm{Fix}(\varphi)\cap \mathcal{P}$ contains at least one edge. Then it must fix some colours and exchange others. Let $A = \Psi(\varphi) \in GL_k^{\mathrm{or}}(\mathbb{Z}_2)$ be the isomorphism induced by $\varphi$. Since $A^2=Id$, up  to equivalences, the canonical vectors can be chosen to be colours on faces or around edges in $\mathrm{Fix}(\varphi)\cap \mathcal{P}$, and
$$A=\bigoplus_s\begin{pmatrix}
0 & 1\\
1 & 0
\end{pmatrix} \oplus \bigoplus_{k-2s} (1) \, \, \text{for some} \, \, 0<s<\frac{k}{2}. $$

Now, $A$ acts on $\mathrm{Row} \Lambda$ by taking $v^t \Lambda$ to $v^t A \Lambda$. Therefore, $A$ fixes $\omega=x^t \Lambda \Leftrightarrow x^t A=x^t \Leftrightarrow A^t x = x \Leftrightarrow A x = x$. Moreover, we have the correspondence $$\{K_\omega \in K_\Omega\} \leftrightarrow \big\{ \omega = v^t\Lambda \in \mathrm{Row} \Lambda \setminus \{0,\epsilon\}\big\} \leftrightarrow \big\{v \in \mathbb{Z}_2^k \setminus \{0,(1, \ldots, 1)\}\big\}, $$  
and we have that  $F^* \in K_\omega$ with $\omega=v^T \Lambda$ if and only if $\lambda(F) \cdot v =1$.

Then, let us take a right-angled $n$-gon $P$ and a bijection $\theta$ between the sides of $\mathrm{Fix}(\varphi)\cap \mathcal{P}$ and the sides of $P$. Again, by \cite[Theorem 2.4]{Vesnin}, $n\ge 5$.  Equip $P$ with the colouring $\mu$ into the colouring space $\mathbb{Z}_2^s \oplus \mathbb{Z}_2^{k-2s}\cong \mathbb{Z}_2^{k-s}$ such that, for an edge $e=\mathrm{Fix}(\varphi)\cap F \in \mathrm{Fix}(\varphi)\cap \mathcal{P}$ belonging to a face coloured with $\lambda(F)=(a_1, \ldots, a_{2s},b_1, \ldots, b_{k-2s})$, $\theta(e)$ is coloured with $\mu\big(\theta(e)\big)=(a_1+a_2, \ldots, a_{2s-1}+a_{2s},b_1, \ldots, b_{k-2s})$. This colouring doesn't depend of the choice of $F$ in case $e$ also belongs to another face $F_1\in \mathcal{F}$, since then $\lambda(F_1)=(a_2,a_1, \ldots, a_{2s},a_{2s-1},b_1, \ldots, b_{k-2s})$. Moreover, $\theta$ takes orientable colours of $\mathbb{Z}_2^k$ into orientable colours of $\mathbb{Z}_2^{k-s}$, since it preserves the sum of entries.

Even if $\mu$ is not proper, $\mathcal{M}_\mu$ is still an orientable $2$-orbifold cover of $P$ of degree $2^{k-s}$, because the canonical basis of $\mathbb{Z}_2^k$ is sent to the canonical basis of $\mathbb{Z}_2^{k-s}$ and thus the colouring is surjective. Moreover, this orbifold has at most $p$ branched points of degree $2$, which cover precisely the points around edges with the same colour. Finally, $p<n$, otherwise all the colours on $P$ would have to be the same and the rank of its colouring would be $1$, while $k-s>\frac{k}{2}\ge \frac{3}{2}$. Since a right-angled $n$-gon has orbifold Euler characterist $\frac{n}{4}-\frac{n}{2}+1=\frac{4-n}{4}$, an orientable orbifold of topological genus $g$ with $p$ degree $2$ singular points has Euler characterist $2-2g-\frac{p}{2}$ and an orbifold covering preserves the orbifold Euler characterist, we have that 
\begin{align*}
    2-2g-\frac{p}{2}=\Big(\frac{4-n}{4}\Big) 2^{k-s} &\Rightarrow 2p = 8+ 2^{k-s}(n-4)-8g < 2n \\
    &\Rightarrow g >1 + 2^{k-s-3}(n-4) -\frac{n}{4} \ge 1+ \frac{n-4}{2} - \frac{n}{4}=\frac{n-4}{4}>0.
\end{align*}

We thus have that the orbifold cover of $P$ is not a topological $QHS$, and it follows from Equation \ref{eq:cohomology} that there exists $\xi \in \mathrm{Row} M$ such that $(\partial P)^*_\xi$ is disconnected. Let $x = (x_1, \ldots, x_s, x_{s+1}, \ldots, x_{k-s}) \in \mathbb{Z}_2^s \oplus \mathbb{Z}_2^{k-2s}\cong \mathbb{Z}_2^{k-s}$ be such that $\xi = x^t M$. Put $y=(x_1, x_1, \ldots, x_s, x_s, x_{s+1}, \ldots, x_{k-s})\in \mathbb{Z}_2^k$ and $\omega = y^t \Lambda \in \mathrm{Row}\Lambda$. Since $Ay=y$, the subcomplex $K_\omega$ is fixed under $\varphi$. Moreover, for $e=\mathrm{Fix}(\varphi)\cap F$,
\begin{align*}
    \lambda(F)\cdot y &= (a_1, a_2, \ldots, a_{2s-1}, a_{2s}, b_1, \ldots, b_{k-2s})\cdot (x_1, x_1, \ldots, x_s, x_s, x_{s+1}, \ldots, x_{k-s})\\
    &= (a_1+a_2, \ldots, a_{2s-1}+a_{2s}, b_1, \ldots, b_{k-2s})\cdot (x_1, \ldots, x_s, x_{s+1}, \ldots, x_{k-s}) = \mu\big(\theta(e)\big) \cdot x
\end{align*}
and it follows that $\theta(e)^* \in (\partial P)^* _\xi\Leftrightarrow F^* \in K_\omega$. As such, $K_\omega \cap \mathrm{Fix}(\varphi)$ is also disconnected. As in the previous case, we conclude by showing that this leads to a contradiction.
\end{proof}

\begin{cor}\label{cor: no good order 2}
Let $\lambda$ be an orientable colouring of $\mathcal{P} \subset \mathbb{X}^3$ such that $\mathcal{M}_\lambda$ is a $QHS$. Then $\mathrm{Sym}_{\lambda}(\mathcal{P})$ does not contain any even-order face rotation or rotoreflections other than the edge rotoreflection.
\end{cor}

\begin{proof}
Let $\varphi \in \mathrm{Sym}_{\lambda}(\mathcal{P})$. If $\varphi$ is an order $2s$ face rotation, then $\varphi^s\in \mathrm{Sym}_{\lambda}(\mathcal{P})$ is an order $2$ face rotation, which would be good by Proposition \ref{prop: no good order 2}, a contradiction.

If, instead, $\varphi$ is a rotoreflection composed of an order $2s$ rotation with $s>2$, then $\varphi^s\in \mathrm{Sym}_{\lambda}(\mathcal{P})$ is an order $2$ face rotation or the antipodal map, depending on whether $s$ is even or odd. In either case, it is a contradiction by  Propositions \ref{prop: no good order 2} and \ref{antipodal symmetries}.

If, finally, $\varphi \in \mathrm{Sym}_{\lambda}(\mathcal{P})$ is a rotoreflection composed of an order $s$ rotation with $s$ odd, then $\varphi^s\in \mathrm{Sym}_{\lambda}(\mathcal{P})$ is a reflection, a contradiction by Proposition \ref{prop: no reflections}.
\end{proof}


\subsection{Admissible Symmetry Groups}

Taking in consideration all the obstructions noted in Propositions \ref{antipodal symmetries}, \ref{prop: no good order 2}, \ref{prop: no reflections} and Corollary \ref{cor: no good order 2}, we can state the following results.

\begin{thm}\label{thm:possible symmetries}
Let $\lambda$ be an orientable colouring of $\mathcal{P} \subset \mathbb{H}^3$ such that $\mathcal{M}_\lambda$ is a $QHS$. Then $\varphi \in \mathrm{Sym}_{\lambda}(\mathcal{P})$ only if $\varphi$ is one of the following:
\begin{enumerate}
    \item a bad edge rotation;
    \item a bad face-edge rotation;
    \item a good edge rotoreflection;
    \item a good odd-order face rotation;
    \item a vertex rotation;
    \item a face-vertex rotation.
\end{enumerate}
In particular, $\mathrm{Sym}_{\lambda}(\mathcal{P})$ does not contain any good order $2$ symmetry, and $\varphi \in \mathrm{Sym}_{\lambda}(\mathcal{P})$ is good only if $\varphi$ is an edge rotoreflection or an odd-order rotation.
\end{thm}

\begin{cor}\label{cor:odd coloured symmetry groups}
Let $\lambda$ be a rank $k\ge 5$ orientable colouring of $\mathcal{P} \subset \mathbb{H}^3$ such that $\mathcal{M}_\lambda$ is a $QHS$. Then any $\varphi \in \mathrm{Sym}_{\lambda}(\mathcal{P})$ is an odd-order rotation. In particular, $|\mathrm{Sym}_{\lambda}(\mathcal{P})|$ is odd.
\end{cor}

\begin{proof}
By Proposition \ref{prop:is good}, any order $2$ rotation would be good for $k \ge 5$, a contradiction to the theorem. Edge rotoreflections are excluded too, because their squares are edge rotations. The only remaining symmetries are odd-order rotations.
\end{proof}

If we fix the rank of the colouring, and thus the number of colours, we can also obtain an obstruction to the order of the symmetries and a description of the group of admissible symmetries.

\begin{prop}\label{prop:order preserving}
Let $\lambda$ be a rank $k$ orientable colouring of $\mathcal{P} \subset \mathbb{H}^3$ such that $\mathcal{M}_\lambda$ is a $QHS$ and $\varphi \in \mathrm{Sym}_{\lambda}(\mathcal{P})$. Then $\varphi$ must induce a permutation of colours of order $o(\varphi)$. In particular, $o(\varphi)< 2^{k-1}$, except in the case of admissible edge rotoreflections of small covers.
\end{prop}

\begin{proof}
Let $\varphi \in \mathrm{Sym}_{\lambda}(\mathcal{P})$ be any and $\tilde{\varphi}$ be the permutation of elements $\mathbb{Z}_2^{k,\mathrm{or}}$ induced by $\varphi$. We have that $o(\tilde{\varphi})$ divides $o(\varphi)$. 

Assume $\varphi$ is a face-rotation. If $o(\tilde{\varphi})=1$, then $\tilde{\varphi}$ fixes all the colours, and it follows that $\varphi$ fixes every complex $K_\omega \in K_\Omega$. We can conclude, as in the proof of Proposition \ref{prop: no good order 2}, that the complex not containing the colours of the rotated faces must be fixed by the symmetry, and thus contain a fixed point, an absurd. If, instead, $1<o(\tilde{\varphi})<o(\varphi)$, we can take the non-trivial face rotation $\xi=\varphi^{o(\tilde{\varphi})}\in \mathrm{Sym}_{\lambda}(\mathcal{P})$. By construction, $\tilde{\xi}$ fixes all the colours, again an absurd. Thus, we must have that $o(\tilde{\varphi})=o(\varphi)$. 

Otherwise, if $\varphi$ is a (face--) edge rotation or a (face--) vertex rotation, then its order is $2$ or $3$ respectively and $\tilde{\varphi}$ cannot fix all the colours, from which it follows that $o(\tilde{\varphi})=o(\varphi)$. Finally, if $\varphi$ is an edge rotoreflection and $o(\tilde{\varphi})<o(\varphi)=4$, then $\varphi^2$ is an edge reflection that fixes all colours, an absurd since the colours around the edge must be exchanged.

Since $|\mathbb{Z}_2^{k,\mathrm{or}}| = 2^{k-1} \ge 4$, the only admissible symmetries $\varphi$ with $o(\varphi)>4$ are face rotations by Theorem \ref{thm:possible symmetries}, and a face rotation must fix at least one colour, the last result follows.
\end{proof}

\begin{thm}\label{thm:admissible groups}
Let $\lambda$ be a rank $k$ orientable colouring of $\mathcal{P} \subset \mathbb{H}^3$ such that $\mathcal{M}_\lambda$ is a $QHS$. Then $\mathrm{Sym}_{\lambda}(\mathcal{P})$ is isomorphic to a subgroup of $GL_k^{\mathrm{or}}(\mathbb{Z}_2)$. In particular, $\mathrm{Sym}_{\lambda}(\mathcal{P})\hookrightarrow S_4$ for rank $3$ covers (small covers) and $\mathrm{Sym}_{\lambda}(\mathcal{P})\hookrightarrow\big(\mathbb{Z}_2^3 \rtimes SL_3(\mathbb{Z}_2)\big)$ for rank $4$ colourings.
\end{thm}

\begin{proof}
The natural homomorphism $\Psi:\mathrm{Sym}_{\lambda}(\mathcal{P})\to GL_k(\mathbb{Z}_2)$, defined in Subsection~\ref{subsection:symmetries}, can be naturally restricted to one into $GL_k^{\mathrm{or}}(\mathbb{Z}_2)$ in the case of orientable colourings. By Proposition \ref{prop:order preserving}, this homomorphism is order preserving, therefore injective. As such, the general result follows. We also have that $GL_3^{\mathrm{or}}(\mathbb{Z}_2)\cong S_4$ and $GL_4^{\mathrm{or}}(\mathbb{Z}_2) \cong \mathbb{Z}_2 ^3 \rtimes GL_3(\mathbb{Z}_2)$, from which the results for low-rank colourings follow as well.
\end{proof}

\begin{rem}
$\Psi$ doesn't need to be injective in general. Indeed, the canonical colouring of a polytope $\mathcal{P}$, assigning a different canonical vector to each facet of $\mathcal{P}$, has $\mathrm{Sym}_{\lambda}(\mathcal{P}) \cong \mathrm{Sym}(\mathcal{P})$, and this group may be larger than $GL_k^{\mathrm{or}}(\mathbb{Z}_2)$.
\end{rem}

\subsection{Examples}

\begin{example}[Tetrahedron]
The canonical colouring of the right-angled tetrahedron, assigning a distinct canonical vector to each facet, produces the $3$-sphere, which is naturally a $QHS^3$. By construction, every symmetry of the tetrahedron is admissible, and its admissible symmetry group is $S_4$. We note that this group contains reflections, but this is not contradictory since the right-angled tetrahedron is an elliptic orbifold.
\end{example}

\begin{example}[$3$-cube]
The euclidean $3$-cube has a unique rank $4$ colouring class that produces a flat $QHS^3$ known as \textit{the Hantzsche--Wendt manifold} \cite{HW} (see \cite[Section 3.1]{FKS} for its colouring representation).  In this case, the admissible symmetry group is $S_3$.
\end{example}

\begin{example}[Dodecahedron]\label{ex:dodecahedron}
The only orientable small cover of the right-angled dodecahedron $\mathcal{D}$, induced by its $4$-colouring, has admissible symmetry group $A_4$ \cite[p. 6]{GS}. This is largest possible group for an orientable small cover of the dodecahedron, since $\mathrm{Sym}(\mathcal{D})\cong A_5$ and $GL_3^\mathrm{or}(\mathbb{Z}_2)\cong S_4$.

The dodecahedron also admits $44$ $QHS^3$ colourings of rank $4$, up to equivalence. Of those, $25$ have trivial admissible symmetry group, $14$ have $\mathbb{Z}_2$ group, two have $\mathbb{Z}_2 \times \mathbb{Z}_2$ group, other two have $\mathbb{Z}_3$ group and one has $S_3$ group. Again, since $GL_4^\mathrm{or}(\mathbb{Z}_2)\cong SL(3,2)\ltimes \mathbb{Z}_2 ^3$, we have that $\mathrm{Adm}_\lambda (\mathcal{P}) < A_4$. Moreover, order $2$ elements necessarily correspond to edge rotations and order $3$ elements to vertex rotations. The colouring $\lambda$ such that $\mathrm{Sym}_{\lambda}(\mathcal{P})\cong S_3$ is given by the matrix below, and the face labelling of the dodecahedron is given in Figure \ref{im:labels}.
\begin{equation*}
\begin{pmatrix}
1 & 0 & 0 & 1 & 0 & 0 & 1 & 0 & 1 & 1 & 1 & 0 \\
0 & 1 & 0 & 1 & 0 & 1 & 0 & 0 & 0 & 0 & 1 & 0 \\
0 & 0 & 1 & 0 & 0 & 0 & 1 & 1 & 0 & 1 & 0 & 0 \\
0 & 0 & 0 & 1 & 1 & 0 & 1 & 0 & 0 & 1 & 1 & 1 \\
\end{pmatrix}
\end{equation*}
\end{example}

\begin{figure}[h]       
    \fbox{\includegraphics[scale=0.2]{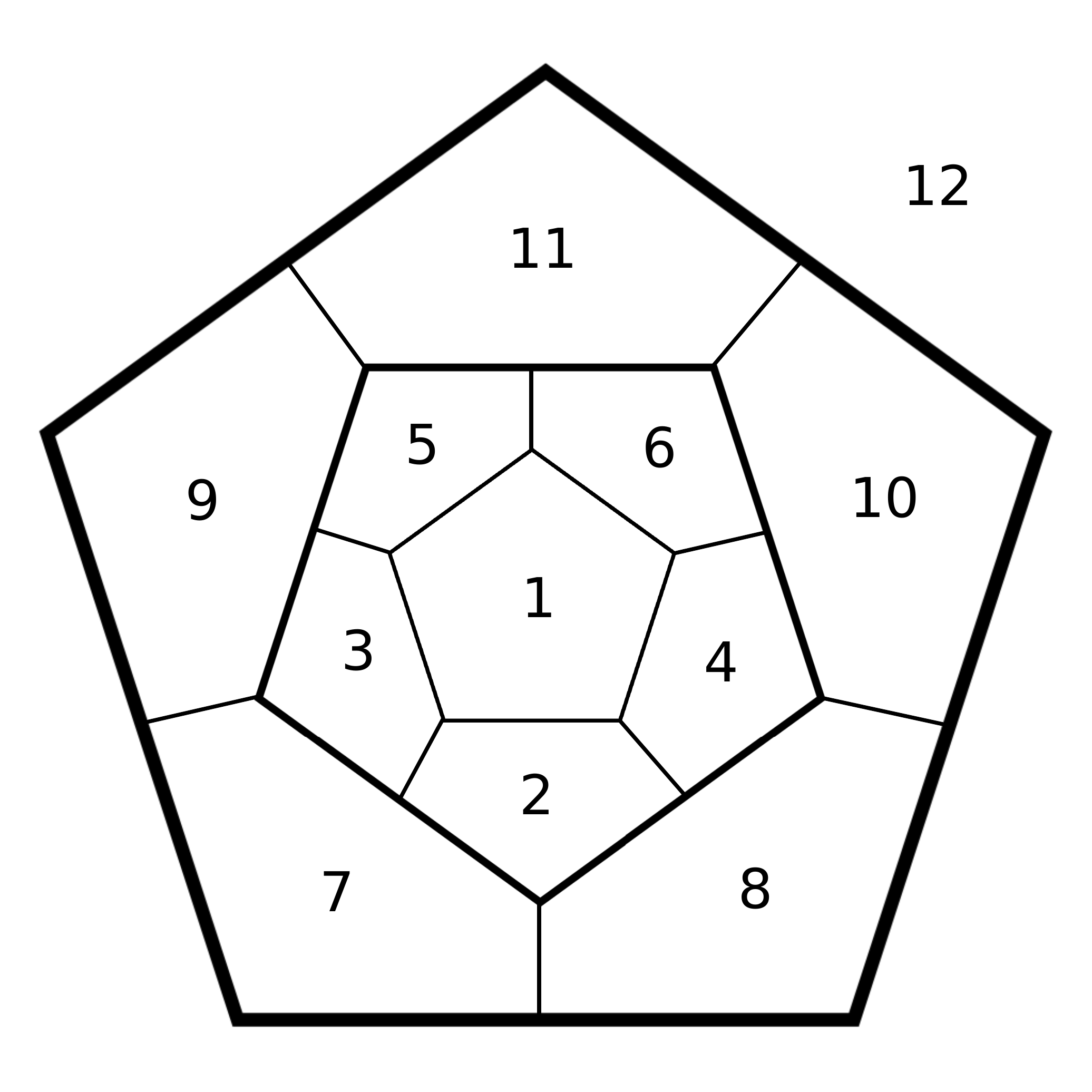}}   
    \hspace{30px}
    \fbox{\includegraphics[scale=0.2]{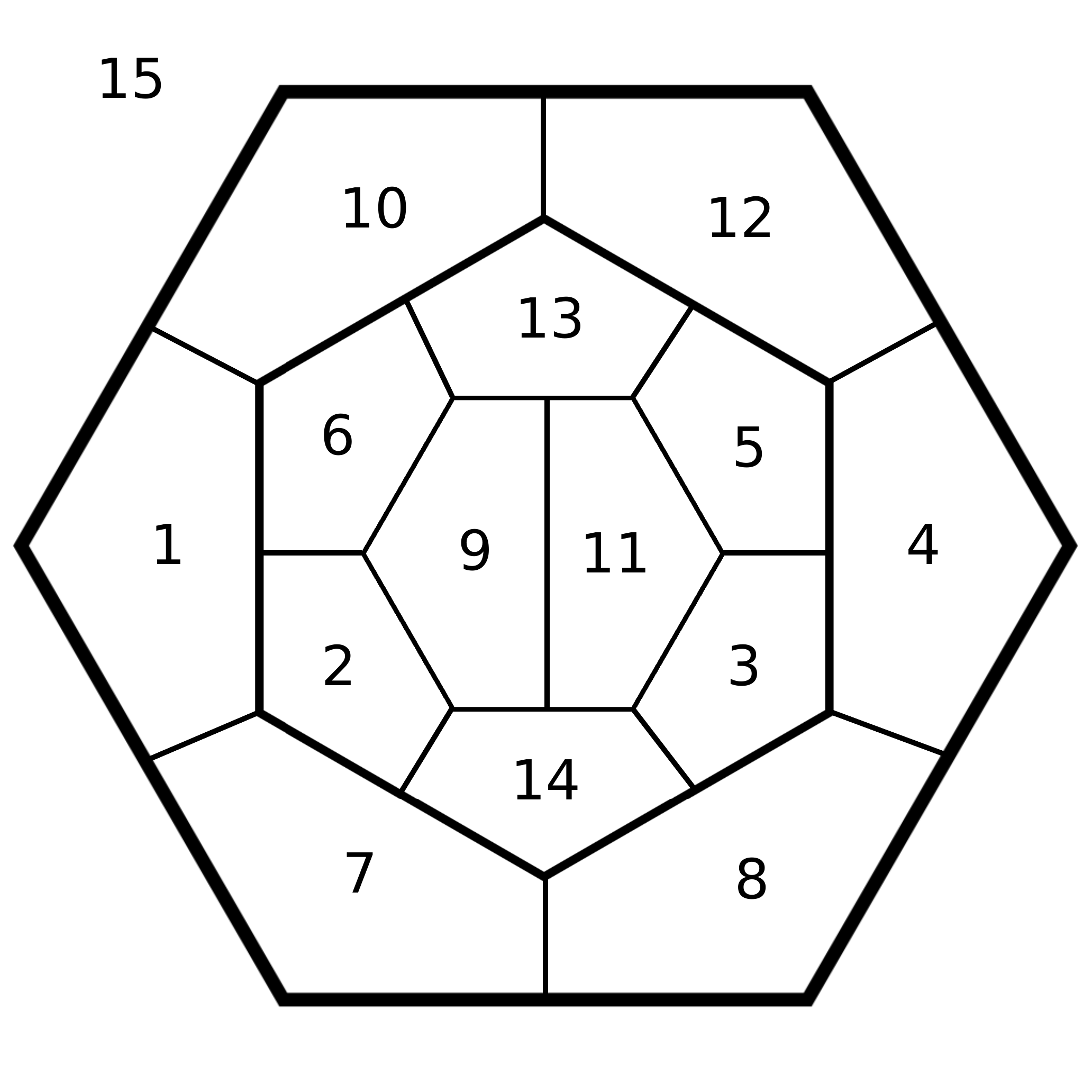}}
    \hspace{30px}
    \fbox{\includegraphics[scale=0.2]{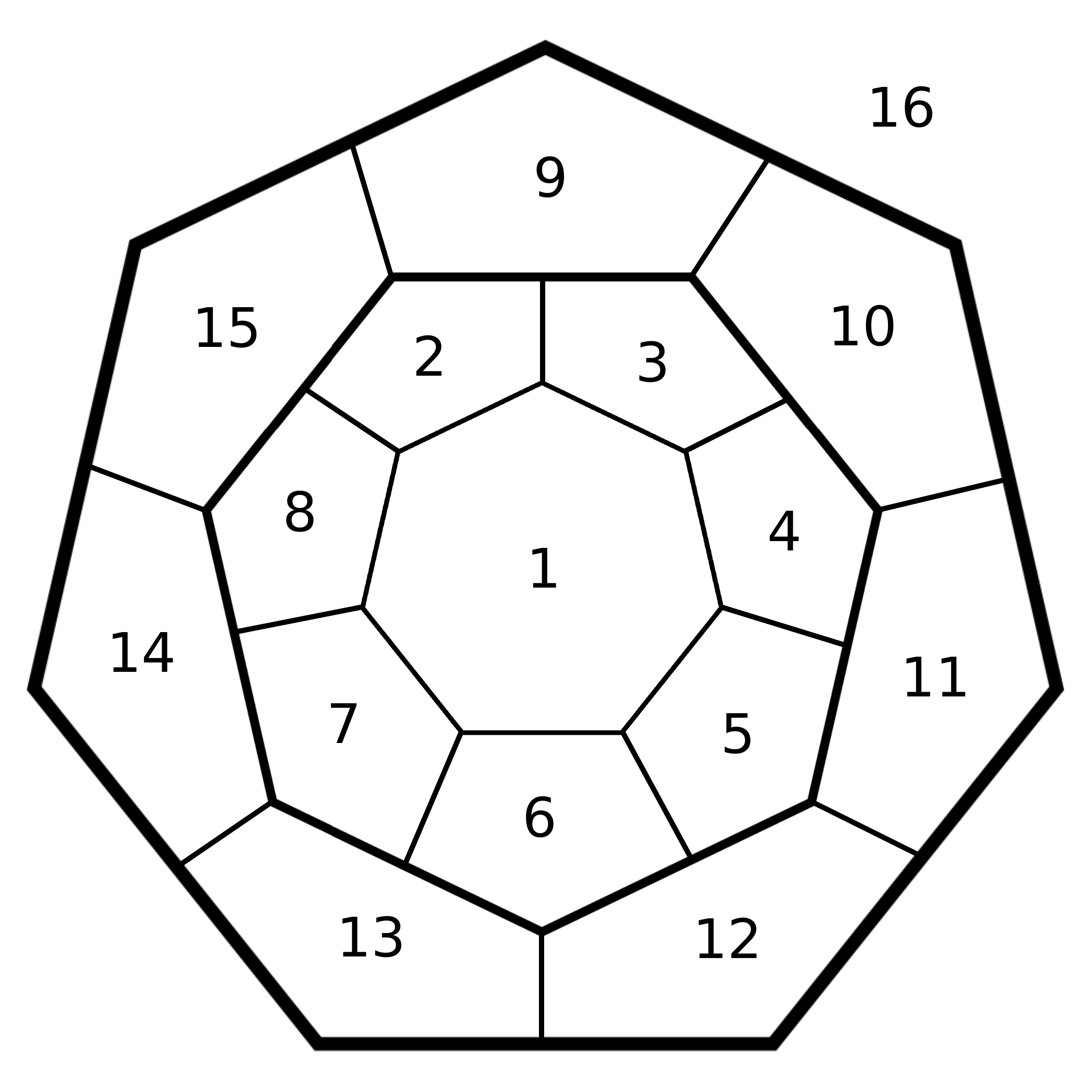}}
    \caption{Left to right: face labellings of the dodecahedron $R(5)$, $R(6)_1^1$ and $R(7)$, respectively}
    \label{im:labels}
\end{figure}

\begin{example}[L\"obell polyhedra]\label{ex:Lobell}
Consider the hyperbolic, right--angled L\"obell polyhedron $R(N)$, defined in \cite[p. 342]{Vesnin}. If $N = 2 \mod 3$, then $R(N)$ admits a small cover $\mathcal{M}_\lambda$ which is a $QHS^3$. Indeed, it is enough to label the opposite $N$-gons with colours $1$ and $2$, and colour the inner and outer rings cyclically with the colours indicated by the red and blue blocks in Figure \ref{im:Lobell}. The other colours (in black) on the inner and outer rings are not repeated. Now we can easily see that the subgraphs $G_{12}$, $G_{13}$ and $G_{23}$ are trees, and by Corollary \ref{cor:small cover QHS subgraphs are trees} this colouring is a $QHS^3$. 
\end{example}

\begin{figure}[h]
    \centering
    \includegraphics[scale=0.3]{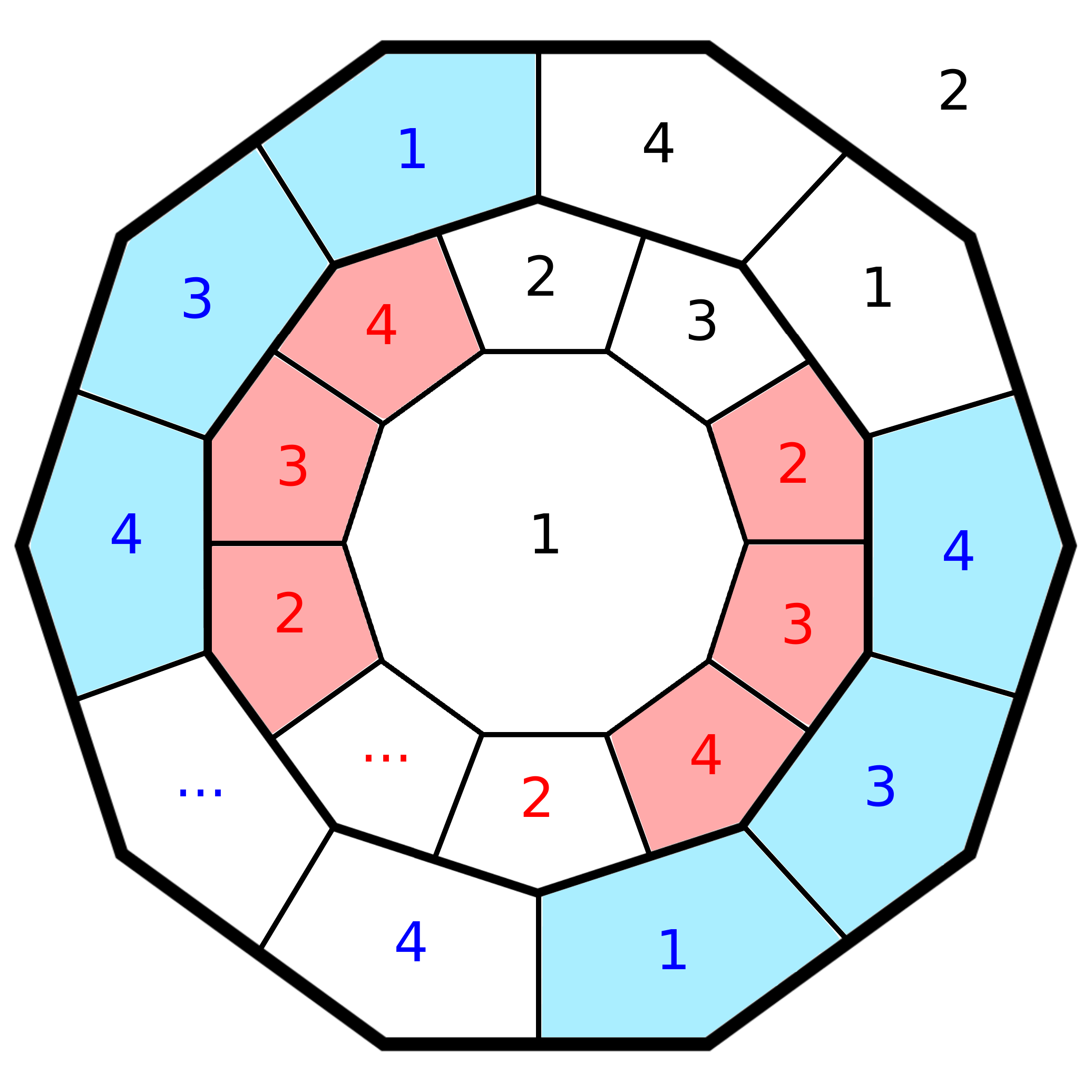}
    \caption{The $4$-colouring of the L\"obell polyhedron $R(N)$ defined in Example \ref{ex:Lobell}}
    \label{im:Lobell}
\end{figure}

\begin{example}[An admissible face-edge rotation]\label{ex:face-edge-rotation}
Let $R(6)_1^1$ be the L\"obell polyhedron $R(6)$ with an edge surgery performed in one of the hexagonal faces as described in \cite[pp. 343-344]{Vesnin}. This polyhedron admits a hyperbolic, right-angled realization, and it has a unique $QHS^3$ small cover $\mathcal{M}_\lambda$. Moreover, we have that $\mathrm{Sym}_\lambda (\mathcal{P}) \cong S_3$, where the order $3$ elements are vertex rotations and the order $2$ elements are face-edge rotations. The colouring $\lambda$ is given by the matrix below, and the face labelling of $R(6)_1 ^1$ is given in Figure~\ref{im:labels}.
\begin{equation*}
\begin{pmatrix}
1 & 0 & 0 & 0 & 1 & 0 & 0 & 1 & 1 & 0 & 0 & 0 & 1 & 1 & 1 \\
0 & 1 & 0 & 1 & 0 & 0 & 0 & 0 & 0 & 1 & 1 & 0 & 1 & 1 & 1 \\
0 & 0 & 1 & 0 & 0 & 1 & 1 & 0 & 0 & 0 & 0 & 1 & 1 & 1 & 1 \\
\end{pmatrix}
\end{equation*}
\end{example}

\medskip

\begin{example}[A $\mathbb{Z}_7$ admissible symmetry group]\label{ex:Lobell2}
Consider the Löbell polyhedron $R(7)$, defined in the previous example, with face labelling given by Figure \ref{im:labels}. The $\mathbb{Z}_2^4$-colouring with colouring matrix given below is a hyperbolic $QHS^3$ with $\mathrm{Sym}_{\lambda}(\mathcal{P}) \cong \mathbb{Z}_7$, generated by the $\frac{2\pi}{7}$-rotation of the opposed heptagonal faces. See Proposition \ref{prop:Z7} for more details.
\begin{equation*}
\begin{pmatrix}
1 & 0 & 0 & 0 & 1 & 0 & 1 & 1 & 1 & 0 & 1 & 1 & 0 & 0 & 0 & 1 \\
0 & 1 & 0 & 0 & 1 & 1 & 1 & 0 & 1 & 1 & 1 & 0 & 1 & 0 & 0 & 0 \\
0 & 0 & 1 & 0 & 0 & 1 & 1 & 1 & 0 & 1 & 1 & 1 & 0 & 1 & 0 & 0 \\
0 & 0 & 0 & 1 & 1 & 1 & 0 & 1 & 1 & 1 & 0 & 1 & 0 & 0 & 1 & 0 \\
\end{pmatrix}
\end{equation*}
\end{example}

\medskip

\begin{example}[An admissible edge rotoreflection]\label{ex:fullerene}
Consider fullerene with face labelling given by the vertex labelling of its dual in Figure \ref{im:fullerene}. The small cover with colouring matrix given below is a hyperbolic $QHS^3$ with $\mathrm{Sym}_{\lambda}(\mathcal{P}) \cong \mathbb{Z}_4$, generated by the rotoreflection of the edges dual to $\{12,13\}$ and $\{14,15\}$. This is the unique $QHS^3$ colouring class of that fullerene containing an admissible edge rotoreflection.
\begin{equation*}
\begin{pmatrix}
1 & 0 & 0 & 1 & 1 & 0 & 1 & 0 & 0 & 0 & 1 & 1 & 0 & 0 & 1 & 1 & 1 & 0 & 0 & 1 & 1 & 0 & 1 & 0 & 0 & 1 & 0 & 1\\
0 & 1 & 0 & 1 & 1 & 0 & 0 & 1 & 0 & 1 & 0 & 1 & 1 & 0 & 1 & 0 & 0 & 0 & 1 & 1 & 1 & 1 & 0 & 0 & 0 & 0 & 1 & 1\\
0 & 0 & 1 & 1 & 1 & 1 & 0 & 0 & 1 & 0 & 0 & 1 & 0 & 1 & 1 & 0 & 0 & 1 & 0 & 1 & 1 & 0 & 0 & 1 & 1 & 0 & 0 & 1
\end{pmatrix}
\end{equation*}

\begin{figure}[h]
    \centering
    \includegraphics[scale=0.35]{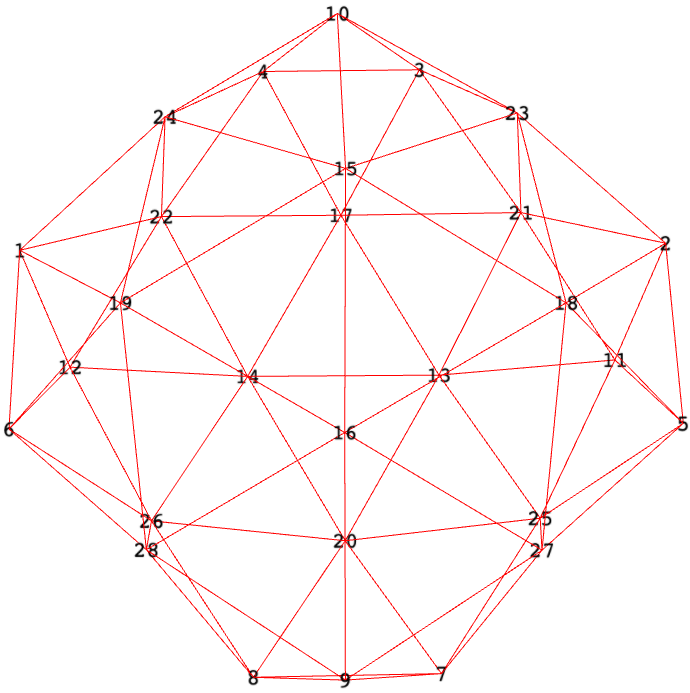}
    \caption{The planar dual to the fullerene from Example~\ref{ex:fullerene}}
    \label{im:fullerene}
\end{figure}
\end{example}

\section{Construction of colourings with given admissible symmetries}\label{sec:constructions}

Knowing what kind of admissible symmetries are useful, we can look for polytopes which contain these symmetries in their symmetry group and then create colourings that are forced to contain them as admissible symmetries.

The general recipe is choosing a symmetry $\varphi \in \mathrm{Sym}(\mathcal{P})$, search for matrices $A \in GL_k(\mathbb{Z}_2)$ of order $o(A)\mid o(\varphi)$ and, having fixed a colour to a facet $F_0\in \mathcal{F}$, set $\lambda\big(\varphi ^i (F_0)\big) = A^i \lambda(F_0)$ for every $1<i<o(\varphi)$. If the obtained colouring is proper, it will contain $\varphi$ as an admissible symmetry.

Moreover, up to equivalences, the choice of colours and matrices can be reduced by fixing some colours as canonical vectors, such as the colours around a vertex for instance. An example of this construction is given below:

\begin{prop}\label{prop:Z7}
There exists a colouring $\lambda$ of a hyperbolic right--angled compact polytope $\mathcal{P}\subset \mathbb{H}^3$ such that $\mathcal{M}_\lambda$ is a $QHS^3$ and $\mathrm{Adm}_{\lambda}(\mathcal{P}) \cong \mathbb{Z}_7$.
\end{prop}

\begin{proof}
Naturally, in order to have an order $7$ element in $\mathrm{Adm}_{\lambda}(\mathcal{P})$, we must have a symmetry with such order in $\mathrm{Adm}(\mathcal{P})$. Also, by Theorem \ref{thm:possible symmetries}, this symmetry must be a face rotation. Therefore $\mathcal{P}$ must contain opposing heptagons and be symmetric under their rotation. In addition, by Theorem \ref{thm:admissible groups}, we must have that the rank colouring is at least $4$, since $S_4$ does not contain elements of order $7$ but $GL_3(\mathbb{Z}_2)$ does.

The natural candidate, therefore, is $R(7)$, one of the Löbell polyhedra (see Example \ref{ex:Lobell2}), with a colouring of rank $4$. Since the face rotation must fix one colour, all the other $7$ colours in $\mathbb{Z}_2^{4,\mathrm{or}}$ must be permuted cyclically. Moreover, the heptagons must be coloured with the fixed colour. Up to equivalences, we can assume this colour is $e_1$. Thus $\lambda(F_1) = \lambda(F_{16}) = e_1$.

We now look at the faces remaining to be coloured, whose labelling can be seen in Figure \ref{im:labels}. Since we have two remaining orbits of faces under the symmetry, the degree of freedom for the colouring is given by the choice of two colours -- one for each orbit -- and the choice of order $7$ matrix $A \in GL_4^{\mathrm{or}}(\mathbb{Z}_2)$. However, up to equivalence, we can still fix a canonical basis around any triangle. So assume we have $\lambda(F_2)=e_2$ and $\lambda(F_3)=e_3$.

It follows that our degrees of freedom now are the choice of a single colour for the second orbit, such as $\lambda(F_9)$ (which must be different than $e_1$), and the choice of $A \in GL_4^{\mathrm{or}}(\mathbb{Z}_2)$ such that $A^7=\mathrm{id}$, $A(e_1) = e_1$ and  $A(e_2) = e_3$. We will then have $\lambda(F_i)=A^{i-3} e_3$ for $4 \le i\le 8$ and $\lambda(F_i)=A^{i-9} \lambda(F_9)$ for $10 \le i \le 15$.

A computer search over the $14$ possible colourings, letting $A$ and $\lambda(F_8)$ vary, provides only one proper colouring class such that $\mathcal{M}_\lambda$ is a $QHS^3$. By construction, we have that $\varphi \in \mathrm{Adm}_{\lambda}(\mathcal{P})$. Moreover, there are no other admissible symmetries in this colouring (besides the powers of $\varphi$), and it follows that $\mathrm{Adm}_{\lambda}(\mathcal{P}) \cong \mathbb{Z}_7$.
\end{proof}

A similar construction can be applied to the polyhedra from Examples \ref{ex:face-edge-rotation} and \ref{ex:fullerene}.

\section{Open Questions}\label{sec:questions}

A few questions related to $QHS^3$ and their symmetries remain to be answered:

\begin{enumerate}
    \item Is there a $QHS^3$ small cover of a hyperbolic, right-angled polytope whose admissible symmetry group is $S_4$ or $D_8$?
    \item Is there a $QHS^3$ rank $4$ colouring of a hyperbolic, right-angled polytope whose admissible symmetry group has order greater than $7$?
\end{enumerate}

\end{document}